\documentclass[11pt]{article}
\usepackage{color}
\definecolor{aleacolour}{rgb}{0.09,0.32,0.44} 
\usepackage[pdfborder={0 0 0},pdfborderstyle={},colorlinks=true,linkcolor=aleacolour,citecolor=aleacolour,urlcolor=blue]{hyperref}
\usepackage[textsize=scriptsize,colorinlistoftodos]{todonotes}

\usepackage{soul}
\usepackage{graphicx,amsthm,amsmath,amssymb,amsfonts}
\usepackage{booktabs}
\usepackage{cases}
\usepackage{fullpage}
\usepackage[small,bf]{caption}
\usepackage[top=2.2cm,bottom=2.2cm,left=2.2cm,right=2.2cm]{geometry}
\usepackage[english]{babel}
\usepackage[utf8x]{inputenc}

\usepackage{amssymb,amsmath, graphicx,enumerate, setspace}
\usepackage{amsthm, mathtools}
\usepackage{supertabular}

\newcommand{\be}{\begin{enumerate}}

\newcommand{\beq}{\begin{equation}}
\newcommand{\eeq}{\end{equation}}
\newcommand{\beqs}{\begin{equation*}}
\newcommand{\eeqs}{\end{equation*}}
\newcommand{\bea}{\begin{eqnarray}}
\newcommand{\eea}{\end{eqnarray}}
\newcommand{\beas}{\begin{eqnarray*}}
\newcommand{\eeas}{\end{eqnarray*}}

\def\({\left(}
\def\){\right)}

\theoremstyle{plain}
\newtheorem{theorem}{Theorem}

\newtheorem{lemma}[theorem]{Lemma}
\newtheorem{claim}[theorem]{Claim}

\newtheorem{conjecture}[theorem]{Conjecture}

\theoremstyle{definition}
\newtheorem{definition}[theorem]{Definition}

\theoremstyle{remark}

\theoremstyle{property}

\title{\vspace{-0.9cm} Threshold Ramsey multiplicity for odd cycles}
\author{David Conlon\thanks{Department of Mathematics, California Institute of Technology, Pasadena, CA 91125. Email: \href{dconlon@caltech.edu} {\nolinkurl{dconlon@caltech.edu}}. Research supported by NSF Award DMS-2054452.}
\and 
Jacob Fox\thanks{Department of Mathematics, Stanford University, Stanford, CA 94305. Email: \href{jacobfox@stanford.edu} {\nolinkurl{jacobfox@stanford.edu}}. Research supported by a Packard Fellowship and by NSF Award DMS-1855635.}
\and 
Benny Sudakov\thanks{Department of Mathematics, ETH, 8092 Z\"urich, Switzerland. Email: \href{mailto:benjamin.sudakov@math.ethz.ch} {\nolinkurl{benjamin.sudakov@math.ethz.ch}}. Research supported by SNSF grant 200021\_196965.}
\and
Fan Wei\thanks{Department of Mathematics, Princeton University, Princeton, NJ 08540. Email: \href{mailto:fanw@princeton.edu} {\nolinkurl{fanw@princeton.edu}}.
Research supported by NSF Award DMS-1953958.
}
}
\date{}

\begin{document}
\maketitle  

\begin{abstract}
The Ramsey number $r(H)$ of a graph $H$ is the minimum $n$ such that any two-coloring of the edges of the complete graph $K_n$ contains a monochromatic copy of $H$. The threshold  Ramsey multiplicity $m(H)$  is then the minimum number of monochromatic copies of $H$ taken over all two-edge-colorings of $K_{r(H)}$. 
The study of this concept was first proposed by Harary and Prins almost fifty years ago. 
In a companion paper, the authors have shown that there is a positive constant $c$ such that the threshold Ramsey multiplicity for a path or even cycle with $k$ vertices is at least $(ck)^k$, which is tight up to the value of $c$. Here, using different methods, we show that the same result also holds for odd cycles with $k$ vertices. 
\end{abstract}

\section{Introduction}

The \emph{Ramsey number} $r(H)$ of a graph $H$ is  the minimum positive integer $n$ such that any two-coloring of the edges of the complete graph $K_n$ on $n$ vertices contains a monochromatic copy of $H$. Ramsey in 1930 proved that these numbers exist. 
However, determining or even estimating Ramsey numbers remains a formidable challenge for most graphs. For instance, the Ramsey number of $K_5$ is already not known, while the longstanding bounds $2^{k/2} \leq r(K_k) \leq 4^k$ have only been improved by lower-order factors~\cite{ConlonUp, Sah, Spencer}. 

To date, there are only a few non-trivial families of graphs for which the Ramsey number is known exactly, including stars, paths, and cycles. 
Let $P_k$ and $C_k$ denote the path and cycle on $k$ vertices, respectively. In 1967, Gerencs\'er and Gy\'arf\'as~\cite{GG} determined the Ramsey number of paths, namely,
\[r(P_k) = k  - 1+ \lfloor k/2 \rfloor.\]
For cycles, the general case was solved independently by Rosta~\cite{Ros} and by Faudree and Schelp~\cite{cycle2}, who showed that
 \[ r(C_k) = 3k/2- 1 \text{ if } k \geq 6 \text{ is even \ \ and \ \ } r(C_k)=2k-1  \text{ if } k \geq 5 \text{ is odd}.\] 

A more general problem than computing Ramsey numbers is to determine the {\it Ramsey multiplicity} $M(H, n)$, the minimum number of monochromatic copies of $H$ guaranteed in any two-edge-coloring of $K_n$. Indeed, it is easy to check that $M(H, n) = 0$ if and only if $n < r(H)$. 

The asymptotic  behaviour of $M(H, n)$ when $H$ is fixed and $n$ tends to infinity has attracted considerable attention. This is in part because of a famous conjecture of Erd\H{o}s~\cite{Erdos62} stating that if $H$ is a clique, then the value of $M(H, n)$ is asymptotically equal to the expected number of monochromatic copies of $H$ in a uniformly random two-edge-coloring of $K_n$. Unfortunately, this conjecture (and a later generalization to all graphs~\cite{BRcommon}) is false already for $H = K_4$, as first shown by Thomason~\cite{Tcommon} (see also~\cite{K4common, Sidcommon}). However, it remains an interesting open problem to determine which graphs satisfy the conjecture, known in the literature as \emph{common graphs}. For instance, the non-three-colorable $5$-wheel is known to be common~\cite{HHKNR} and some hope remains that all bipartite graphs are common because of a connection to a celebrated conjecture of Sidorenko and Erd\H{o}s--Simonovits~\cite{Sidorenko,Sidorenko2, Simon} (see~\cite{CFS, CKLL, CL20, KLL, LS, S1} for some recent results towards this conjecture). We refer the interested reader to~\cite[Section 2.6]{Ramseysurvey} for more on this fascinating subject.
 
Another much-studied problem concerns the  value of $M(H, n)$ when it first becomes positive, i.e., when $n = r(H)$. As in our companion paper \cite{CFSW}, we refer to this value as the \emph{threshold Ramsey multiplicity}. 

\begin{definition}
The \emph{threshold Ramsey multiplicity} $m(H)$ of a graph $H$ is the minimum number of monochromatic copies of $H$ in any two-coloring of the edges of  $K_n$ with $n=r(H)$. In other words, \[m(H) = M(H, r(H)).\]
\end{definition}

The threshold Ramsey multiplicity was first studied systematically by Harary and Prins \cite{HP} almost fifty years ago. The exact value of the threshold Ramsey multiplicity is known for all graphs with at most $4$ vertices~\cite{Hsurvey, HP, K4}, but, in general, determining or even providing a non-trivial lower bound on the threshold Ramsey multiplicity appears to be quite challenging. In fact, the behavior of $m(H)$ can be rather erratic. For instance, Harary and Prins \cite{HP} proved that $m(K_2) = 1$ and  $m(K_{1,k}) =1$ for $k$ even, but $m(K_{1,k})=2k$ for $k \geq 3$ odd. 

In the same paper \cite{HP}, Harary and Prins asked for a determination of $m(P_k)$ and $m(C_k)$. It is this question that concerns us in this paper and its companion~\cite{CFSW}. Indeed, in~\cite{CFSW}, not only did we provide the first non-trivial bound for the Ramsey multiplicity of paths and even cycles, but the bound is tight up to a lower-order factor. 

\begin{theorem}[\cite{CFSW}]\label{thm:even}
There is a positive constant $c$ such that, for every positive integer $k$, the threshold Ramsey multiplicity of the path with $k$ vertices satisfies $m(P_k) \geq (ck)^k$ and, if $k$ is even, the threshold Ramsey multiplicity of the cycle on $k$ vertices satisfies $m(C_k) \geq (ck)^k$.
\end{theorem}

In this paper, we address Harary and Prins' question for odd cycles. Unlike the cases studied in~\cite{CFSW}, this odd-cycle case has received considerable previous attention, with Rosta and Sur\'anyi \cite{RStm} already proving the exponential lower bound $m(C_k)\geq 2^{c k}$ in the 1970's. This was later improved to a superexponential bound in an unpublished work of Rosta (see~\cite{KR}). More recently, K\'arolyi and Rosta~\cite{KR} improved the lower bound to $m(C_k) \geq k^{c k}$.  To the best of our knowledge, this was the state-of-the-art prior to our result, which we now state.

\begin{theorem}\label{thm:main}
There is a positive constant $c$ such that, for every odd positive integer $k$, the threshold Ramsey multiplicity of the cycle on $k$ vertices satisfies $m(C_k) \geq (ck)^k$.
\end{theorem}

As for paths and even cycles, this bound is tight up to the constant $c$. However, it is proved using rather different methods to those employed in~\cite{CFSW}, because the Ramsey numbers, and the associated extremal colorings, are quite different for odd cycles and for paths and even cycles. To describe the extremal colorings in the odd setting, consider the red/blue edge-coloring $\chi(a,b)$ of the complete graph on $n=a+b$ vertices with vertex set $A \cup B$, $|A|=a$ and $|B|=b$, where $A$ and $B$ form blue cliques and all edges between $A$ and $B$ are red. Let $k \geq 5$ be an odd positive integer. Then $\chi(k-1, k-1)$ is a coloring of the complete graph on $r(C_k) - 1 = 2k - 2$ vertices with no monochromatic $C_k$, while $\chi(k, k-1)$ is a coloring of the complete graph on $r(C_k) = 2k-1$ vertices with exactly $(k-1)!/2$ monochromatic copies of $C_k$, as all monochromatic $C_k$ are in the blue clique of order $k$. 
This provides an upper bound on $m(C_k)$ showing that the bound in Theorem~\ref{thm:main} is tight apart from a lower-order factor. It also suggests that our bound can be strengthened, as follows.

 \begin{conjecture} \label{conj:ck} 
 For any sufficiently large odd integer $k$,  $m(C_k) = (k-1)!/2$.
\end{conjecture}


\section{Proof of Theorem \ref{thm:main}}

\subsection{Preliminaries}

As in our proof of Theorem \ref{thm:even}  in \cite{CFSW}, we will use Szemer\'edi's regularity lemma, an important  tool which gives a rough structural decomposition for all graphs. Roughly speaking, for any graph, the regularity lemma outputs a  vertex partition of the graph into a small number of parts, where  the bipartite graph between almost every pair of  parts behaves like a random graph. Among its many applications (see, for example,~\cite{Regularity}), this decomposition is useful for embedding and counting copies of sparse graphs, such as the cycles that concern us here. 

To formally state the regularity lemma, we first need some definitions to quantify what is meant by a ``random-like" bipartite graph. 
For a pair of disjoint vertex subsets $(X, Y)$ of a graph, let $d(X, Y) = e(X,Y)/|X||Y|$
denote the density of edges between $X$ and $Y$.

\begin{definition}[$\epsilon$-regular pair]
A pair $(X, Y)$ of disjoint vertex subsets of a graph is said to be $\epsilon$-regular if, for all subsets $U \subset X, V \subset Y$ such that $|U| \geq \epsilon |X|$ and $|V| \geq \epsilon |Y|$, $|d(U, V) - d(X,Y)| \leq \epsilon$. 
\end{definition}

The next lemma sets out two basic facts about $\epsilon$-regular pairs that will be useful later. 

\begin{lemma}\label{lem:epsregprop}
If $(X,Y)$ is an $\epsilon$-regular pair and $d(X, Y)=d$, then the following hold:
\begin{enumerate}[(i)]
\item If $Y' \subset Y$ satisfies $|Y'| \geq \epsilon |Y|$, then the number of vertices in $X$ with degree in $Y'$ greater than $(d+\epsilon)|Y'|$ is less than $\epsilon |X|$ and the number of vertices in $X$ with degree in $Y'$ less than $(d-\epsilon)|Y'|$ is less than $\epsilon |X|$. 
\item If $X' \subset X$ and $Y'\subset Y$ are such that $|X'| \geq \alpha |X|$ and $|Y'| \geq \alpha |Y|$, then $(X', Y')$ is $\max(\epsilon/\alpha, 2\epsilon)$-regular.
\end{enumerate}
\end{lemma}

A vertex partition is called {\it equitable} if each pair of parts differ in size by at most one. We are now ready to state the regularity lemma.

\begin{lemma}[Szemer\'edi’s regularity lemma]\label{reglem}
For every $\epsilon>0$ and positive integer $l$, there are positive integers $n_0$ and  
$M_0$ such that every graph $G$ with at least $n_0$ vertices admits an equitable vertex partition $V(G) = V_1 \cup \dots \cup V_M$ into $M$ parts with $l \leq M \leq M_0$ where
all but at most $\epsilon \binom{M}{2}$ pairs of parts $(V_i, V_j )$ with  $1 \leq i < j \leq M$ are $\epsilon$-regular.
\end{lemma}

In practice, we will use the following standard colored version of the regularity lemma.

\begin{lemma}[Colored regularity lemma] \label{reglem-twocolor} 
For every $\epsilon>0$ and positive integer $l$, there are positive integers $n_0$ and $M_0$ such that every two-edge-coloring of the complete graph $K_n$ with $n \geq n_0$ in red and blue admits an equitable vertex partition $V(G)=V_1 \cup \dots \cup V_M$ into $M$ parts with $l \leq M \leq M_0$ where all but at most $\epsilon{M \choose 2}$ pairs of parts $(V_i,V_j)$ with $1 \leq i < j \leq M$ are $\epsilon$-regular in both the red and blue subgraphs. 
\end{lemma}

Lemmas \ref{reglem} and \ref{reglem-twocolor} are in fact equivalent, since a pair $(V_i,V_j)$ in an edge-coloring of $K_n$ with colors red and blue is $\epsilon$-regular in red if and only if it is $\epsilon$-regular in blue. 

\subsection{The stability lemma}

The main ingredient in the proof of Theorem~\ref{thm:main} is a stability lemma, Lemma \ref{lem:main2} below, which implies that any two-edge-coloring of $K_n$ (where, for us, $n$ will be $r(C_k) = 2k-1$ for some sufficiently large odd $k$) either has a regularity partition whose reduced graph contains a long monochromatic path or the edge-coloring of $K_n$ is close to the coloring $\chi(k, k-1)$ described before Conjecture \ref{conj:ck}. In either case, we can show that the conclusion of Theorem~\ref{thm:main} must hold.

Before introducing the stability lemma, we need 
to make precise what we mean by saying that a two-edge-coloring of the complete graph on $2k-1$ vertices is close to $\chi(k, k-1)$. In the definition, we will refer to the density of a set $X$, given by $d(X,X) = 2e(X)/|X|^2$. 

\begin{definition}[Extremal coloring with parameter $\lambda$] \label{def:ec2}
A two-edge-coloring of the complete graph on $n$ vertices is an {\it extremal coloring with parameter $\lambda$} 
if there exists a partition $ A \cup B$ of the vertex set such that 
\begin{itemize}
\item $|A| \geq (1/2-\lambda)n$ and $|B| \geq (1/2 - \lambda)n$;
\item the graph induced on $A$ has density at least $(1-\lambda)$ in some color, the graph induced on $B$ has density at least $(1-\lambda)$ in the same color, and the bipartite graph between $A$ and $B$ has density at least $(1-\lambda)$ in the other color. 
\end{itemize}
\end{definition}

Our key stability lemma is now as follows.

\begin{lemma}\label{lem:main2} 
For any $0<\epsilon < 10^{-20}$, there is a constant $M_0=M_0(\epsilon)$ such that if $\alpha = 20 \sqrt{\epsilon}$, 
then, for $n$ sufficiently large in terms of $\epsilon$, any two-edge-coloring of the complete graph $K_n$ falls into one of the following two cases:  
\begin{itemize} 
\item \textbf{Case 1:} There is a positive integer $\epsilon^{-1} \leq M \leq M_0$ 
such that if $t$ is the odd integer with
\begin{equation} 
(1/2 + \alpha) M \geq t > (1/2 + \alpha) M - 2, \label{eqn:reducedcycle}
\end{equation}
then there are disjoint vertex sets $V_0, \dots, V_{t-1}$, indexed by the elements of $\mathbb{Z}/t\mathbb{Z}$, and a color $\chi$ such that, for each $i \in\mathbb{Z} / t \mathbb{Z}$, $|V_i| \geq  \lfloor n/M \rfloor$, the pair $(V_i, V_{i+1})$ is $\epsilon$-regular in the color $\chi$, and the edge density between $V_i$ and $V_{i+1}$ in the color $\chi$ is at least $11\epsilon^{1/2}$.
\item  \textbf{Case 2:} The graph is an extremal coloring with parameter $300 \sqrt{\alpha}$.
\end{itemize} 
\end{lemma}

We will hold off on proving Lemma~\ref{lem:main2} until Section~\ref{sec:stab}, first showing, across the next two sections, how Theorem~\ref{thm:main} follows from either of the conclusions in the lemma.

\subsection{Theorem~\ref{thm:main} for colorings satisfying Case 1 of Lemma~\ref{lem:main2}} \label{subsec:case1odd}

In this section, we prove Theorem~\ref{thm:main} for colorings satisfying Case 1 of Lemma~\ref{lem:main2}. We will repeatedly work in a setting where we have disjoint vertex sets $V_0, \dots, V_{t-1}$ from a graph where the indices of the $V_i$ are the elements of $\mathbb{Z} / t \mathbb{Z}$. We say that a path $P$ of length $\ell$ with vertices $w_0, w_1, \dots, w_\ell$ and edges $w_0 w_1, w_1 w_2, \dots, w_{\ell - 1} w_\ell$ is {\it $(V_0, \dots, V_{t-1})$-transversal} if $w_i \in V_i$ for each $0 \leq i \leq \ell$. Note that we will typically have $\ell > t$, so the path may pass through each of the vertex sets multiple times.

\begin{lemma}\label{lem:countpath2}
Suppose that $0<\epsilon < 10^{-5}$ and $t$ and $n$ are integers with $t\geq 2$ and $n \geq \epsilon^{-2}$. Suppose also that $V_0, \dots, V_{t-1}$ are disjoint vertex sets in a graph where the indices of the $V_i$ are the elements of $\mathbb{Z} / t \mathbb{Z}$
and, for each $i \in\mathbb{Z} / t \mathbb{Z}$, $|V_i| \geq n$, $(V_i, V_{i+1})$ is  $\epsilon$-regular, and $d(V_i,V_{i+1}) \geq d$ for some $d  \geq 5\epsilon^{1/2}$. Then the following hold:  
\begin{enumerate}
\item For any integer $\ell$ with $2\leq \ell \leq t (1-\sqrt{\epsilon}) n$ and any vertex $w_0 \in V_0$ with at least $(d-\epsilon)|V_1|$ neighbors in  $V_1$, the number of $(V_0, \dots, V_{t-1})$-transversal paths of length $\ell$ starting from $w_0$ is at least $(d-\epsilon - \sqrt{\epsilon})^\ell \prod_{i=1}^\ell (n-\lfloor i/t \rfloor)$.
\item For any integer $\ell$ with $4 \leq \ell \leq t (1-3\sqrt{\epsilon}) n$ which is divisible by $t$ and any two (not necessarily distinct) vertices $w_0, w_0' \in V_0$ such that $w_0$ has at least $(d-\epsilon)|V_1|$ neighbors in  $V_1$ and $w'_0$ has at least $(d-\epsilon)|V_{t-1}|$ neighbors  $V_{t-1}$,  the number of $(V_0, \dots, V_{t-1})$-transversal paths of length $\ell$ with end vertices $w_0$ and $w_0'$ is at least $(d-5\sqrt{\epsilon})^{\ell-1} (1-2\sqrt{\epsilon})^{\ell-2} (\epsilon n)\prod_{i=1}^{\ell-2} (n-\lfloor i/t \rfloor).$
\end{enumerate}
\end{lemma}

\begin{proof} 
For any integer $0 \leq l\leq \ell$, let $N_l$ be the number of good paths of length $l$ starting from $w_0$, where a $(V_0, \dots, V_{t-1})$-transversal path $w_0, w_1, \dots, w_l$ of length $l$ starting from $w_0$ is {\it good} if there are at least $(d-\epsilon)\big(|V_{l+1}| - \lfloor (l+1)/t \rfloor\big)$ ways to extend the path to $V_{l+1}$. We will prove by induction that $N_l \geq (d-\epsilon - \sqrt{\epsilon})^l \prod_{i=1}^l (n-\lfloor i/t \rfloor)$ for $0 \leq l \leq \ell$, which will settle Part 1.  

For the base case, note that $N_0=1$, since the path with zero edges starting from $w_0$ is $w_0$ itself and, by assumption, the vertex $w_0 \in V_0$ has at least $(d-\epsilon)|V_1|$ neighbors in $V_1$. Suppose now that the required lower bound holds for $N_{l-1}$ and we wish to deduce the lower bound for $N_l$.

Fix any good path $P$ of length $l-1$. By the definition of goodness, there are at least $(d - \epsilon) (|V_{l}| - \lfloor l/t \rfloor)$ choices of $w_l \in V_l$ that extend $P$. To bound $N_l$, we need a lower bound on the number of vertices $w_l \in V_l$ such that the path formed by extending $P$ to $w_l$ is also good. 

Let $U$ be the set of vertices in $V_{l}$ which have degree less than $(d - \epsilon) |V_{l+1} \setminus V(P)| = (d - \epsilon) (|V_{l+1}| - \lfloor (l+1)/t \rfloor)$ in $V_{l+1} \setminus  V(P)$. Note that, since $\ell \leq t (1-\sqrt{\epsilon}) n$,
\begin{align*}
|V_{l+1} \setminus V(P)| & =    |V_{l+1}| - \lfloor (l+1)/t \rfloor \geq   |V_{l+1}| -  \ell /t - 1  \geq  |V_{l+1}| -  t(1-\sqrt{\epsilon})n/t - 1\\
  & \geq  |V_{l+1}| -  (1-\sqrt{\epsilon})|V_{l+1}|  - 1 =  \sqrt{\epsilon} |V_{l+1}| - 1 \geq \epsilon |V_{l+1}|.
\end{align*}
Together with the fact that $(V_{l}, V_{l+1})$ is $\epsilon$-regular with density at least $d$, Lemma~\ref{lem:epsregprop} (i) implies that  
$|U| \leq \epsilon |V_{l}|$. 
Hence, the number of choices for $w_{l}$ such that $P$ extended to $w_l$ is also good is at least 
\[ (d-\epsilon)(|V_l| - \lfloor l/t \rfloor) - |U| \geq  (d-\epsilon)(|V_l| - \lfloor l/t \rfloor) - \epsilon |V_l| \geq  (d-\epsilon - \sqrt{\epsilon})(|V_l| - \lfloor l/t \rfloor). 
\] 
The last inequality is equivalent to $(\sqrt{\epsilon} - \epsilon) |V_l| \geq  \sqrt{\epsilon} \lfloor \ell/t \rfloor$, which again follows from $\ell \leq t (1-\sqrt{\epsilon}) n$. Thus,
\begin{equation*}
N_l \geq (d-\epsilon-\sqrt{\epsilon})  (n-\lfloor l/t \rfloor) N_{l-1} \geq (d-\epsilon-\sqrt{\epsilon})^l \prod_{i=1}^l (n-\lfloor i/t \rfloor),  
\end{equation*}
establishing Part 1.

For Part 2, we pass from the $V_i$ to a collection of subsets $V'_i$. By assumption, $|N_{V_{t-1}}(w_0')|\geq (d-\epsilon) |V_{t-1}|\geq \epsilon |V_{t-1}|$, so we may set aside a subset $W_{t-1}$ of $N_{V_{t-1}}(w_0')$ of size $\epsilon |V_{t-1}|$ 
and let $V'_{t-1} = V_{t-1} \setminus W_{t-1}$. If $w'_0$ is distinct from $w_0$, we let $V'_0 = V_0 \setminus \{w'_0\}$, while, in all remaining cases, we let $V'_i = V_i$, noting that $|V'_i| \geq (1 - \epsilon) |V_i|$ for all $i$.
Therefore, if we set $\epsilon' = 2\epsilon$ and $d' = d-\epsilon$, Lemma~\ref{lem:epsregprop} (ii) now tells us that for each $i \in \mathbb{Z} /t\mathbb{Z}$ the pair of sets $(V_i', V_{i+1}')$ is $\epsilon'$-regular with edge density at least $d'$.

As in Part 1, for any positive integer $0 \leq l \leq \ell-3$, let $N_l$ be the number of good paths of length $l$ starting from $w_0$, though with the condition now reading that there are at least $(d'-\epsilon')(|V_{l+1}'| - \lfloor (l+1)/t \rfloor)$ ways to extend the path to a vertex in $V_{l+1}'$. Since $w_0$ has at least $(d-\epsilon)|V_1| - |V_1 \setminus V'_1| \geq (d-2\epsilon) |V_1| > (d'-\epsilon') |V'_1|$ neighbors in $V'_1$ and $\ell \leq t(1-3\sqrt{\epsilon})n \leq t(1-\sqrt{2\epsilon})((1-\epsilon)n) = t(1-\sqrt{\epsilon'}) ((1-\epsilon)n) $, we may apply Part 1 to conclude that
\begin{equation}
N_{\ell-3} \geq  (d'-\epsilon'-\sqrt{\epsilon'})^{\ell-3} \prod_{i=1}^{\ell-3} ((1-\epsilon)n-\lfloor i/t \rfloor).  \label{eqn:p1Nl3}
\end{equation}

Fix any such path $P$ of length $\ell-3$. 
Suppose its vertices are $w_0, w_1, \dots, w_{\ell-3}$ in order, where $w_j \in V_j'$. By definition,  there are at least  $(d'-\epsilon')(|V_{\ell-2}'| - \lfloor (\ell-2)/t \rfloor)$ ways to extend the path to a vertex $w_{\ell-2} \in V_{\ell-2}'$. Denote this set  of candidates for $w_{\ell-2}$ by $C$. Since $\ell$ is divisible by $t$, we have that $C \subset V_{t-2}$. Using that $d \geq 5\sqrt{\epsilon}$ and $\ell/t \leq (1-3\sqrt{\epsilon})n$, we have that $|C| \geq  (d'-\epsilon')(|V_{\ell-2}'|- \lfloor (\ell-2)/t \rfloor) 
\geq \epsilon' |V_{t-2}'| \geq \epsilon |V_{t-2}|$. Since $|W_{t-1}| \geq \epsilon |V_{t-1}|$ and $(V_{t-2}, V_{t-1})$ is $\epsilon$-regular with density at least $d$, the number of edges between $C$ and $W_{t-1}$ is at least 
\begin{equation}
 (d - \epsilon)  |C||W_{t-1}| \geq (d-\epsilon)  (d'-\epsilon')(|V_{\ell-2}'|- \lfloor (\ell-2)/t \rfloor)  \cdot \epsilon |V_{t-1}|. \label{eqn:CVb}
\end{equation}

Note now that $P$ together with the two end vertices of any edge in $E(C, W_{t-1})$ results in a path of length $\ell-1$ that can be extended to $w_0'$. 
Therefore, using \eqref{eqn:p1Nl3} and \eqref{eqn:CVb}, we see that the number of paths with end vertices $w_0$ and $w_0'$ is at least 
\begin{align*}
N_{\ell-3}\cdot   (d - \epsilon)  |C||W_{t-1}| 
& \geq (d'-\epsilon'-\sqrt{\epsilon'})^{\ell-3} \prod_{i=1}^{\ell-3} ((1-\epsilon)n-\lfloor i/t \rfloor)  \\
& \ \ \ \ \ \ \cdot (d-\epsilon)  (d'-\epsilon')(|V_{\ell-2}'|- \lfloor (\ell-2)/t \rfloor)  \cdot \epsilon |V_{t-1}| \\
& \geq  
(d-3\epsilon-2\sqrt{\epsilon})^{\ell-1} (\epsilon n) \prod_{i=1}^{\ell-2} ((1-\epsilon)n-\lfloor i/t \rfloor) \\
& \geq  (d-3\epsilon-2\sqrt{\epsilon})^{\ell-1}(1-\epsilon - \sqrt{\epsilon} )^{\ell-2} (\epsilon n) \prod_{i=1}^{\ell-2} (n-\lfloor i/t \rfloor) \\
& \geq  
(d-5\sqrt{\epsilon})^{\ell-1} (1-2\sqrt{\epsilon})^{\ell-2} (\epsilon n)\prod_{i=1}^{\ell-2} (n-\lfloor i/t \rfloor),
\end{align*}
where the second to last inequality holds because $(1-\epsilon)n - \lfloor i/t \rfloor  \geq (1-\epsilon - \sqrt{\epsilon}) (n - \lfloor i/t \rfloor )$ for $i \leq t(1-2\sqrt{\epsilon})n$. 
\end{proof}

Part 2 of Lemma \ref{lem:countpath2} with $w_0 = w_0'$ implies that there are many cycles of length $\ell$ when $\ell$ is divisible by $t$. 
The next lemma shows that the same result holds even when $\ell$ is not divisible by $t$. 

\begin{lemma}\label{lem:countcycle1}
Suppose that $0< \epsilon <10^{-5}$, $t$ is an odd integer with $t\geq 3$, and $n$ is a positive integer with $n \geq t\epsilon^{-2}$. Suppose also that $V_0, \dots, V_{t-1}$ are disjoint vertex sets in a graph where the indices of the $V_i$ are the elements of $\mathbb{Z}/t\mathbb{Z}$ and, for each $i\in \mathbb{Z}/t\mathbb{Z}$, $|V_i| \geq n$, $(V_i, V_{i+1})$ is $\epsilon$-regular, and $d(V_i,V_{i+1}) \geq d$  for some $d \geq 10\epsilon^{1/2}$. Then, for any odd positive integer $p$ with
\begin{equation} 
2t+6 \leq p \leq t(1-5\sqrt{\epsilon})n, \label{eqn:q}
\end{equation}
the number of cycles of length $p$ is at least  $\frac{\epsilon^2}{4}  n^4  (d-10\sqrt{\epsilon})^{p-2}(1-3\sqrt{\epsilon})^{2p} \prod_{i=1}^{p-4} (n-\lfloor i/t \rfloor)$.
\end{lemma}
\begin{proof}
Suppose $p \equiv r \bmod t$ for some $0 \leq r \leq t-1$. If $r = 0$, then we can choose $w_0=w_0'$ in Part 2 of Lemma \ref{lem:countpath2} in at least $(1-2\epsilon)n$ ways. Then this lemma easily implies that the number of cycles of length $p$ is at least $(1-2\epsilon)\epsilon n^2 (d-5\sqrt{\epsilon})^{p-1} (1-2\sqrt{\epsilon})^{p-2} \prod_{i=1}^{p-2} (n-\lfloor i/t \rfloor)$, which is stronger than the required bound.
We may therefore assume that $p$ is not divisible by $t$, i.e., that $r >0$. 
Define an auxiliary constant $h$ by $h= r+t$ if $r$ is odd, $h = r$ if $r > 2$ and even, and $h =2+2t$ if $r = 2$. Note that $h$ is always a positive even integer with $4 \leq h \leq  2t+ 2$, so a cycle of length $p$ can be constructed by combining an even path $L_1$ of length $h$ alternating between $V_0$ and $V_{t-1}$ with a $(V_0, \dots, V_{t-1})$-transversal path $L_2$ of length $p-h$ with the same end vertices as $L_1$. Since $(p-h)$ is divisible by $t$, the path  $L_2$ will use exactly $(p-h)/t$ vertices in each of the $V_i$. 

By Lemma \ref{lem:epsregprop} (i), the set $U_0$ of vertices in $V_0$ with at least $(d -\epsilon)|V_1|$ neighbors in $V_1$ and at least $(d -\epsilon)|V_{t-1}|$ neighbors in $V_{t-1}$ has size at least $(1-2\epsilon)|V_0|$. We now fix two vertices $u, v \in U_0$ and bound the number of paths of the types $L_1$ and $L_2$ described above with end vertices $u$ and $v$.

We first bound the number of paths of type $L_1$ with end vertices $u$ and $v$. Since $h \geq 4$ and $h \leq 2t + 2 \leq 2(1-3\sqrt{\epsilon})n$, we may apply Part 2 of Lemma \ref{lem:countpath2} to $(V_0, V_{t-1})$ to conclude that the number of paths of length $h$ alternating between $V_0$ and $V_{t-1}$ with end vertices $u$ and $v$ is at least 
\begin{equation}
(d-5\sqrt{\epsilon})^{h-1} (1-2\sqrt{\epsilon})^{h-2} (\epsilon n)\prod_{i=1}^{h-2} (n-\lfloor i/2 \rfloor). \label{eq:part1}
\end{equation}

We now bound the number of paths of type $L_2$ available for each such $L_1$. 
For a given $L_1$, remove its $h-1$ interior vertices from $V_0$ and $ V_{t-1}$, calling the updated vertex sets $V'_0$ and $V'_{t-1}$, respectively. Since $h \leq 2t+2$, we have  $|V_0'| \geq |V_0| - (t+1) \geq (1-\epsilon/2)|V_0|$ and, similarly, $|V_{t-1}'| \geq (1-\epsilon/2)|V_{t-1}|$. Hence, by Lemma \ref{lem:epsregprop} (ii), each of the pairs $(V_0', V_{t-1}')$, $(V_0', V_1)$, and $(V_{t-1}', V_{t-2})$ is $2\epsilon$-regular with density at least $d-\epsilon$. Furthermore, $u$ and $v$ each have at least $(d-\epsilon)|V_{t-1}| - (t+1) \geq (d-2\epsilon)|V_{t-1}| \geq (d-2\epsilon)|V_{t-1}'|$ neighbors in $V_{t-1}'$. 
Since also
\[ 4 \leq p-h \leq t(1-5\sqrt{\epsilon})n \leq t(1-3\sqrt{2\epsilon})(1-\epsilon/2)n,\]
we may apply Part 2 of Lemma \ref{lem:countpath2} with $\epsilon$, $\ell$, $d$, and $n$ replaced by $2\epsilon$, $p-h$, $d-\epsilon$, and $(1-\epsilon/2)n$, respectively, to conclude that the number of $(V'_0, V_1, \dots, V_{t-2}, V'_{t-1})$-transversal paths of length $p-h$ with end vertices $u$ and $v$, each of which is a valid choice for $L_2$, is at least 
\begin{align} 
& (d-\epsilon - 5\sqrt{2\epsilon})^{p-h-1} (1-2\sqrt{2\epsilon})^{p-h-2} (2\epsilon (1-\epsilon/2)n)\prod_{i=1}^{p-h-2} ((1-\epsilon/2)n-\lfloor i/t \rfloor) \nonumber \\ 
& \geq  (d-10\sqrt{\epsilon})^{p-h-1} (1-3\sqrt{\epsilon})^{p-h-2} (\epsilon n)\prod_{i=1}^{p-h-2} ((1-\epsilon/2)n-\lfloor i/t \rfloor) \nonumber\\
& \geq (d-10\sqrt{\epsilon})^{p-h-1} (1-3\sqrt{\epsilon})^{p-h-2} (\epsilon n) (1-\epsilon/2 - \sqrt{\epsilon/2})^{p-h-2} \prod_{i=1}^{p-h-2} (n-\lfloor i/t \rfloor) \nonumber \\
& \geq (d-10\sqrt{\epsilon})^{p-h-1} (1-3\sqrt{\epsilon})^{p-h-2} (\epsilon n) (1- 2\sqrt{\epsilon})^{p-h-2} \prod_{i=1}^{p-h-2} (n-\lfloor i/t \rfloor) \nonumber \\
& \geq (d-10\sqrt{\epsilon})^{p-h-1} (1-3\sqrt{\epsilon})^{2(p-h-2)} (\epsilon n)  \prod_{i=1}^{p-h-2} (n-\lfloor i/t \rfloor),
\label{eq:part2}
\end{align}
where the second inequality holds because $(1-\epsilon/2)n - \lfloor i/t \rfloor  \geq (1-\epsilon/2 - \sqrt{\epsilon/2}) (n - \lfloor i/t \rfloor )$ for $i \leq t(1-2\sqrt{\epsilon/2})n$. 
 
Therefore, since the number of choices for $u$ and $v$ is at least $\frac{1}{2} (1-2\epsilon)|V_0| ((1-2\epsilon)|V_0|-1)$, we may combine \eqref{eq:part1} and \eqref{eq:part2} to conclude that the number of cycles of length $p$ is at least
 \begin{align*}
  & \frac{1}{2} (1-2\epsilon)|V_0| ((1-2\epsilon)|V_0|-1) \cdot (d-5\sqrt{\epsilon})^{h-1} (1-2\sqrt{\epsilon})^{h-2} (\epsilon n)\prod_{i=1}^{h-2} (n-\lfloor i/2 \rfloor)  \\
  & \ \ \ \ \ \cdot (d-10\sqrt{\epsilon})^{p-h-1} (1-3\sqrt{\epsilon})^{2(p-h-2)} (\epsilon n)  \prod_{i=1}^{p-h-2} (n-\lfloor i/t \rfloor) \\
  &\geq  \frac{1}{4} (1-2\epsilon)^2 n^2  (d-10\sqrt{\epsilon})^{p-2}(1-3\sqrt{\epsilon})^{2p-h-2}(\epsilon n)^2  \prod_{i=1}^{p-4} (n-\lfloor i/t \rfloor) \\
  & \geq  \frac{\epsilon^2}{4}  n^4  (d-10\sqrt{\epsilon})^{p-2}(1-3\sqrt{\epsilon})^{2p} \prod_{i=1}^{p-4} (n-\lfloor i/t \rfloor),
 \end{align*}
as required.
\end{proof}

Finally, we can show that Theorem~\ref{thm:main} holds for colorings satisfying Case 1 of Lemma \ref{lem:main2}.

\begin{proof}[Proof of Theorem \ref{thm:main} for colorings satisfying Case 1 of Lemma \ref{lem:main2}.]
Suppose, for concreteness, that $\epsilon = 10^{-30}$ and let $V_0, \dots, V_{t-1}$ be as in Case 1 of Lemma~\ref{lem:main2} with $n = 2k - 1$, where $k$ (and hence $n$) is a sufficiently large odd integer. 
We wish to apply Lemma~\ref{lem:countcycle1} to show there are many cycles of length $k = (n+1)/2$. 
To confirm that the conditions of Lemma \ref{lem:countcycle1} hold, we only have to check \eqref{eqn:q}, i.e., that
\begin{equation}
2t+6 \leq k \leq t(1-5\sqrt{\epsilon}) \lfloor n/M \rfloor.\label{eqn:knM}
\end{equation} 
By \eqref{eqn:reducedcycle}, $(1/2 + \alpha) M \geq t > (1/2 + \alpha) M - 2$, so it suffices to show that
\[k = (n+1)/2 \geq 2(1/2+\alpha)M+6\] 
and
\[ k \leq \left(  (1/2 + \alpha) M - 2 \right)  (1-5\sqrt{\epsilon})(n/M-1).\] 
The first inequality easily holds for $n$ sufficiently large in terms of $\epsilon$, while the second inequality holds because 
\begin{align*}
k = (n+1)/2 \leq  (1+\epsilon)n/2  \leq (1/2 + \alpha/2)M (1-5\sqrt{\epsilon})(1-\epsilon)(n/M), 
\end{align*}
where we used that $\alpha=20\sqrt{\epsilon}$ and 
$
(1+\alpha)(1-5\sqrt{\epsilon})(1-\epsilon) \geq (1 + 20\sqrt{\epsilon})(1-6\sqrt{\epsilon}) >  1+ \epsilon
$.
Hence, since $M \geq 2/\alpha$ and assuming $n \geq M/\epsilon$,
\begin{align*}
k \leq &  (1/2 + \alpha/2)M (1-5\sqrt{\epsilon})(1-\epsilon)(n/M) 
\leq  \left(  (1/2 + \alpha) M - 2 \right)  (1-5\sqrt{\epsilon})(n/M-1),
\end{align*}
as required.

 We may therefore apply Lemma \ref{lem:countcycle1} with parameters $\epsilon$, $p$, $d$, and $n$ replaced by $\epsilon$, $k$, $d$, and $\lfloor n/M \rfloor$, respectively, to conclude that the number of cycles of length $k$ is at least 
 \begin{equation} 
 \frac{\epsilon^2}{4}  ( \lfloor n/M \rfloor)^4  (d-10\sqrt{\epsilon})^{k-2}(1-3\sqrt{\epsilon})^{2k} \prod_{i=1}^{k-4} ( \lfloor n/M \rfloor-\lfloor i/t \rfloor). \label {eq:cyclek}
 \end{equation}
Since $ \lfloor n/M \rfloor \geq k/t$  by  (\ref{eqn:knM}), the last term in 
 (\ref{eq:cyclek}) satisfies
\begin{equation*}
\prod_{i=1}^{k-4} ( \lfloor n/M \rfloor-\lfloor i/t \rfloor)  \geq
\prod_{i=1}^{k-4} ( \lfloor n/M \rfloor- i/t) \geq
t^{-(k-4)} (k-4)!.
  \end{equation*} 
Therefore, (\ref{eq:cyclek}) is lower bounded by 
  \begin{align*}
 &  \frac{\epsilon^2}{4}  ( \lfloor n/M \rfloor)^4  (d-10\sqrt{\epsilon})^{k-2}(1-3\sqrt{\epsilon})^{2k} t^{-(k-4)}(k-4)! \\
 & \geq \frac{\epsilon^2}{4} \frac{n^4}{(2M)^4} (d-10\sqrt{\epsilon})^{k-2}(1-3\sqrt{\epsilon})^{2k}t^{-(k-4)}  ((k-4)/e)^{k-4}  \\
 & \geq \frac{\epsilon^2}{4(2M)^4}  (d-10\sqrt{\epsilon})^{k-2}(1-3\sqrt{\epsilon})^{2k}t^{-(k-4)}  ((k-4)/e)^{k}.
  \end{align*}
Hence, since $\epsilon$ is a constant, $d \geq 11 \sqrt{\epsilon}$, and $M$ and $t$ are bounded in terms of $\epsilon$, there is a constant $c_1$ depending only on $\epsilon$ such that the number of cycles of length $k$ is at least $(c_1 k)^k$, as required.
\end{proof}

\subsection{Theorem~\ref{thm:main} for colorings satisfying Case 2 of Lemma~\ref{lem:main2}} \label{subsec:case2odd}

The proof of Theorem~\ref{thm:main} for colorings satisfying Case 2 of Lemma~\ref{lem:main2} has several cases. To make the presentation cleaner, we first prove some simple claims.

\begin{claim}\label{claim:redgeneral}
Let $S$ and $T$ be two disjoint vertex sets in a graph $F$ such that any two vertices in $S$ have at least $s$ common neighbors in $T$. If there is an edge within $S$, then the number of cycles of length $l$ is at least $\left( s-l/2 +3/2\right)^{(l-1)/2}  (|S| - (l-1)/2)^{(l-3)/2}$ for any odd integer $3 \leq l \leq \min(2s+1,  2|S|-1)$.
\end{claim}

\begin{proof}
Suppose that $(u,u')$ is an edge in $S$. To construct cycles of length $l$, we will find paths $P$ of the form $u, v_1, u_1, v_2, u_2, \dots, v_{(l-1)/2}, u_{(l-1)/2} = u'$ alternating between $S$ and $T$ with end vertices $u$ and $u'$. Each such path together with the edge $(u,u')$ clearly gives rise to a cycle of length $l$. 

To estimate the number of paths $P$ of the required form, 
suppose that $u, u_1, u_2, \dots, u_{(l-1)/2-1}$, $u_{(l-1)/2} = u'$ is a fixed sequence of distinct vertices in $S$. Note that any of the $s$ common neighbors of $u$ and $u_1$ in $T$ can be chosen as $v_1$.  More generally, given choices for $v_1, \dots, v_{i-1}$, the number of choices for $v_i$ is at least $s - (i-1)$ for $1 \leq i \leq (l-1)/2$, since we may pick any vertex in the common neighborhood of $u_{i-1}$ and $u_{i}$ in $T$ except $v_1, \dots, v_{i-1}$. Therefore, given $u, u_1, u_2, \dots, u_{(l-1)/2-1}$, $u_{(l-1)/2} = u'$, the number of choices for $P$ is at least 
\[ \prod_{i=1}^{(l-1)/2} (s - (i-1)) \geq  \left( s-((l-1)/2-1)) \right)^{(l-1)/2} = \left( s-l/2 +3/2 \right)^{(l-1)/2} .\]
Since the number of choices for $u_1, \dots, u_{(l-1)/2-1}$ is 
\[(|S|-2)(|S|-3) \dots (|S| - (l-1)/2) \geq (|S| - (l-1)/2)^{(l-3)/2},\] 
the claim follows.
\end{proof}

\begin{claim}\label{claim:ABedgegeneral}
Let  $S$ and $T$ be two disjoint cliques in a graph $F$. Suppose that there are two vertex-disjoint paths $P_1$ and $P_2$ such that each path has length at most $2$, has one end vertex in $S$, the other end vertex in $T$, and the rest of the vertices outside $S \cup T$. Then the number of cycles of length $l$ is at least $(((l-1)/2-3)/e)^{l-6}$ for any integer $7 \leq l \leq \min(2|S|-1, 2|T|-1)$.
\end{claim}

\begin{proof}
Suppose that  $P_1$ has length $l_1$ and endpoints $a_1 \in S$ and $b_1 \in T$, while $P_2$ has length $l_2$ and endpoints $a_2 \in S$ and $b_2 \in T$. 
Let $s = \lfloor l/2 \rfloor$. We will construct cycles of length $l$ by concatenating the following four paths:  (1) a path $L_1$ of length $s - l_1$ in $S$ with end vertices $a_1$ and $a_2$, (2) the path $P_2$, (3) a path $L_2$ in $T$ of length $l-s-l_2$ with end vertices $b_1$ and $b_2$, and (4) the path $P_1$. This process clearly yields a cycle of length $l$. 

Since $S$ is a clique, we can always find paths $L_1$ of length $s-l_1$ in $S\setminus \{a_1, a_2\}$ with end vertices $a_1$ and $a_2$. 
Indeed, the number of such $L_1$ is exactly the number of length-$(s-l_1-1)$ ordered sequences of vertices in $S \setminus \{a_1, a_2\}$. Since $|S| - 2 \geq s-l_1 -1$, such a sequence exists. Furthermore, the number of such sequences is  exactly $(|S| - 2)!/(|S| - 2 -(s-l_1-1))! \geq  (s-l_1-1)! \geq ((s-l_1-1)/e)^{s-l_1-1}$, where we used the inequality $(x-1)\cdots (x-y) \geq y! \geq (y/e)^y$ for positive integers $x, y$ with $x \geq y+1$. Similarly, the number of choices for $L_2$ is at least $((l-s-l_2-1)/e)^{l-s-l_2-1}$. In total, the number of cycles of length $l$ is at least
$ ((s-l_1-1)/e)^{s-l_1-1} ((l-s-l_2-1)/e)^{l-s-l_2-1}$. Since $l_1, l_2 \leq 2$ and $l-s\geq s$, the quantity above is at least $((s-3)/e)^{s-l_1-1}  ((s-3)/e)^{l-s-l_2-1}$. If $s-3 \geq e$, then the previous quantity is at least $((s-3)/e)^{l-6}$. Otherwise, we counted a positive integer number of cycles of length $l$, which is at least the bound in the claim. 
\end{proof}

\begin{claim}\label{claim:PbothABgeneral}
Let $S$ and $T$ be two disjoint vertex sets in a graph $F$. Suppose that $w \in S$ and there is a complete bipartite graph between $S \setminus \{w\}$ and $T$ and at least one edge between $w$ and $T$ (so, in particular, there may be a complete bipartite graph between $S$ and $T$). If there is a path $P'$ of length two with one end vertex in $S$, the other end vertex in $T$, and the rest of the vertices outside $S \cup T$, then the number of  cycles of length $l$ is at least $((l-5)/2e)^{l-5}$ for any odd integer $l$ with $7 \leq l \leq \min(2|S|+1, 2|T|+1)$. 
\end{claim}

\begin{proof}
Suppose the two end vertices of $P'$ are $a\in S$ and $b\in T$. We will construct  cycles of length $l$ by concatenating $P'$ with paths $P$ of length $l-2$ alternating between $S$ and $T$ with end vertices $a$ and $b$. Fix a neighbor $x$ of $w$ in $T$. If $a = w$, each path $P$ will start with $w$ and then $x$ before returning to some $a' \in S$, while if $a \neq w$, we will avoid $w$ while building our paths. In either case, a lower estimate for the number of cycles of length $l$ is given by estimating the number of paths of length $l - 4$ starting at a fixed $a' \neq w$ and ending at $b$ alternating between $S \setminus \{w\}$ and $T \setminus \{x\}$.

Since there is a complete bipartite graph between $S \setminus \{w\}$ and $T \setminus \{x\}$, any length-$(l-5)/2$ sequence of ordered  vertices in $S\setminus \{a', w\}$  and any length-$(l-5)/2$ sequence of ordered vertices in $T \setminus \{b, x\}$ give rise to a relevant path by alternating between the two sequences as interior vertices. Such sequences exist because $|S| -2 \geq (l-5)/2$ and $|T| - 2 \geq (l-5)/2$.
Thus, the number of choices for the path is at least the product of the number of such sequences in 
$S \setminus \{a', w\}$ and $T \setminus \{b, x\}$, which is 
\begin{align*} 
\frac{(|S|-2)!}{(|S| - (l-5)/2) - 2)!} \cdot \frac{(|T|-2)!}{(|T| - (l-5)/2) - 2)!}
& \geq  ((l-5)/2)! ((l-5)/2)! \\
& \geq ((l-5)/2e)^{l-5},
\end{align*}
where we again used that $(x-1)\cdots (x-k) \geq k! \geq (k/e)^k$ for positive integers $x, k$ with $x \geq k+1$. 
\end{proof}

\begin{claim}\label{claim:tworedABgeneral}
Let $S$ and $T$ be two disjoint vertex sets in a graph $F$. If there are no two vertex-disjoint edges between $S$ and $T$, then, by removing at most one vertex from $S\cup T$, there is no edge between $S$ and $T$.
\end{claim}

\begin{proof}
If there is no vertex in $S$ with a neighbor in $T$, the claim trivially holds. 
If there is exactly one vertex $a \in S$ with neighbors in $T$, then there is no edge between $S \setminus \{a\}$ and $T$. If there is more than one vertex in $S$ with a neighbor in $T$, then all of them have the same neighbor $b \in T$ and there is no edge between $S$ and $T \setminus \{b\}$. In each case, the claim follows.
\end{proof}

We are now ready to prove Theorem \ref{thm:main} for colorings satisfying Case 2 of Lemma \ref{lem:main2}.

\begin{proof}[Proof of Theorem \ref{thm:main} for colorings satisfying Case 2 of Lemma \ref{lem:main2}.]
Suppose again that $\epsilon = 10^{-30}$ and $n = 2k - 1$ for $k$ a sufficiently large odd integer, but we now have an extremal coloring of $K_n$ with parameter $\lambda = 300\sqrt{\alpha}$ and vertex partition $A \cup B$, as in Case 2 of Lemma~\ref{lem:main2}.
Without loss of generality, we will assume that the red densities within $A$ and $B$ are both at least $1-\lambda$ and the blue density between $A$ and $B$ is at least $1-\lambda$, where $|A|, |B| \geq (1/2 - \lambda)n$. 

We first conduct a simple cleaning-up procedure. 

\begin{claim}
\label{claim:updateEC}
There is a vertex partition of $K_n$ as $A'\cup B' \cup X \cup Y$ satisfying the following conditions:
\begin{enumerate}
\item $A = A' \cup X$ and $B = B' \cup Y$.
\item $|X | \leq 2\sqrt{\lambda}|A|$ and $|Y| \leq 2\sqrt{\lambda}|B|$.
\item $|A'| \geq (1/2 - 2\sqrt{\lambda})n$ and $|B'| \geq (1/2 - 2\sqrt{\lambda})n$. 
\item Each vertex in $A'$ has red degree at least $ (1-3\sqrt{\lambda})|A|$ in $A'$ and blue degree at least $ (1-3\sqrt{\lambda})|B|$ in $B'$. 
Similarly, each vertex in $B'$ has red degree at least $ (1-3\sqrt{\lambda})|B|$ in $B'$ and blue degree at least $ (1-3\sqrt{\lambda})|A|$ in $A'$. 
\end{enumerate} 
\end{claim}

\begin{proof}
Suppose that there are $x|A|$ vertices in $A$ whose red degree in $A$ is at most $(1-\sqrt{\lambda})|A|$. Then,  
\[
x|A|(1-\sqrt{\lambda})|A| + (1-x)|A||A| \geq (1-\lambda) |A|^2,
\]
which implies that $x \leq \sqrt{\lambda}$. 
Similarly, there are at most $\sqrt{\lambda}|A|$ vertices in $A$ whose blue degree in $B$ is at most $(1-\sqrt{\lambda})|B|$. 
Letting $X$ be the union of these two bad sets of vertices, we see that $|X | \leq 2\sqrt{\lambda}|A|$. We define $Y \subset B$ similarly, again noting that $|Y| \leq 2\sqrt{\lambda}|B|$. Letting $A' = A \setminus X$ and $B' = B \setminus Y$, we see that items 1 and 2 hold.
To verify item 3, note that $|A'| = |A| - |X| \geq (1-2\sqrt{\lambda})|A|$. Since $|A| \geq (1/2 - \lambda) n$, we have 
$$|A'| \geq (1-2\sqrt{\lambda})(1/2-\lambda)n \geq (1/2 - 2\sqrt{\lambda})n,$$ 
as required. Similarly, $|B'| \geq (1/2 - 2\sqrt{\lambda})n$.
Finally, for item 4, note, for example, that each vertex in $A' = A \setminus X$ has red degree at least $(1-\sqrt{\lambda})|A| - |X| \geq (1-3\sqrt{\lambda})|A|$ in $A'$, while each vertex in $A'$ has blue degree at least $(1-\sqrt{\lambda})|B| - |Y| \geq (1-3\sqrt{\lambda})|B|$ in $B'$. 
\end{proof}

The following claim allows us to assume that all the edges in $A'$ and $B'$ are red, i.e., that $A'$ and $B'$ are both red cliques, as otherwise we would be done. 

\begin{claim}\label{claim:red}
If there is a blue edge within $A'$ or $B'$, then the number of blue cycles of length $k$  with $k = (n+1)/2$ is at least $(n/5)^{k-2}.$
\end{claim}

\begin{proof}
Without loss of generality, suppose that there is a blue edge $(u,u')$ in $A'$. 
We will apply Claim \ref{claim:redgeneral} to the blue subgraph with $S = A'$, $T = B'$, and $l = k$. Since, by Claim \ref{claim:updateEC}, every vertex in $A'$ has at least $(1-3\sqrt{\lambda})|B|$ blue neighbors in $B'$, the size of the common blue neighborhood in $B'$ of any two vertices in $A'$ is at least \[2(1-3\sqrt{\lambda})|B| - |B'| \geq (1-6\sqrt{\lambda})|B'|.\]
Thus, again in reference to Claim \ref{claim:redgeneral}, we may take $s = (1-6\sqrt{\lambda})|B'|$. 

It remains to verify that the conditions of Claim \ref{claim:redgeneral} hold, that is, that $k \leq \min(2(1-6\sqrt{\lambda})|B'|+1, 2|A'|-1)$. 
But this is simple, since
\beq (1-6\sqrt{\lambda})|B'| - (k-1)/2 \geq (1-6\sqrt{\lambda})(1/2 - 2\sqrt{\lambda})n - n/4 > (1/4 - 5\sqrt{\lambda})n \label{eq:eq1}
\eeq
 and
 \beq |A'| - (k+1)/2 \geq (1/2 - 2\sqrt{\lambda})n - (k+1)/2 \geq  (1/4 - 5\sqrt{\lambda})n.\label{eq:eq2}
 \eeq 
Therefore, by Claim \ref{claim:redgeneral} and the estimates \eqref{eq:eq1} and \eqref{eq:eq2},  the number of cycles of length $k$ is at least
\[
((1/4 - 5\sqrt{\lambda})n)^{(k-1)/2} \cdot ((1/4 - 5\sqrt{\lambda})n)^{(k-3)/2}  = ((1/4 - 5\sqrt{\lambda})n)^{k-2} \geq (n/5)^{k-2}, 
\]
as required.
\end{proof}

Our next claim is as follows. 

\begin{claim}\label{claim:ABedge}
Suppose that $A'$ and $B'$ are both red cliques. If there are two vertex-disjoint red paths $P_1$ and $P_2$ such that each has  length at most $2$ and each has one end vertex in $A'$ and the other in $B'$, then there are at least $(n/5e)^{k-6}$ red cycles of length $k$. 
\end{claim}

\begin{proof}
We will apply Claim \ref{claim:ABedgegeneral} to the red subgraph with $S = A'$, $T = B'$, and $l = k$. 
To check that the condition $7 \leq k \leq \min(2|A'| -1, 2|B'|-1)$ holds, note that
$$2|A'|-1 \geq (1 - 4\sqrt{\lambda})n-1 > (n+1)/2 = k$$ 
for $n$ sufficiently large and, similarly, $k \leq 2|B'|-1$. Therefore, we may apply Claim~\ref{claim:ABedgegeneral} to conclude that the number of red cycles of length $k$ is at least 
$$(((k-1)/2-3)/e)^{k-6}=\left( ((n-1)/4 - 3)/e\right)^{k-6} \geq (n/5e)^{k-6}$$ 
for $n$ sufficiently large.
\end{proof}

Therefore, 
we are done if the assumptions of Claim \ref{claim:ABedge} are satisfied, so we can and will assume that if $A'$ and $B'$ are both red cliques, then there are no two vertex-disjoint red edges between $A'$ and $B'$. 
By applying Claim \ref{claim:tworedABgeneral} to the red subgraph with $S = A'$ and $T = B'$, we see that we may remove at most one vertex from $A' \cup B'$ to make all the edges between $A'$ and $B'$ blue. Without loss of generality, we may therefore assume that there is a vertex $v$ in $A'$ such that all the edges between $A' \setminus \{v\}$ and $B'$ are blue. In what follows, we let $A'' = A' \setminus \{v\}$. 

\begin{claim}\label{claim:PbothAB}
Suppose that all the edges between $A''$ and $B'$ are blue. If there is a blue path $P'$ of length two with one end vertex in $A''$ and the other in $B'$, then there are at least $(n/8e)^{k-5}$ blue cycles of length $k$. 
\end{claim}

\begin{proof}
We will apply Claim \ref{claim:PbothABgeneral} to the blue subgraph with $S = A''$, $T = B'$, and $l = k$, using the fact that the bipartite graph between $A''$ and $B'$ is complete in blue. To check that the condition $7 \leq k \leq \min(2|A''|+1, 2|B'|+1)$, note, for instance, that 
$$2|A''| + 1 \geq 2|A'| - 1 \geq (1 - 4 \sqrt{\lambda})n -1 \geq (n+1)/2 = k$$ 
for $n$ sufficiently large. Therefore, by Claim~\ref{claim:PbothABgeneral}, the number of blue cycles of length $k$ is at least 
$((k-5)/2e)^{k-5} \geq  (n/8e)^{k-5}$, as required.
\end{proof}

Since we are done if the assumptions of Claim \ref{claim:PbothAB} hold, we can now assume that there is no blue path $P'$ of length two with one end vertex in $A''$ and the other in $B'$. This means that any vertex in $\{v\} \cup X \cup Y$ is either completely red to $A''$ or completely red to $B'$. Therefore, there is a vertex partition of $\{v\} \cup X \cup Y$ into two sets  $Z_1 \cup Z_2$ such that each vertex in $Z_1$ is completely red to $A''$ and each vertex in $Z_2$ is completely red to $B'$.  

By another application of Claim \ref{claim:ABedge}, we can also assume that there are no two vertex-disjoint red paths of length at most two each with one end vertex in $A''$ and the other in $B'$. Therefore, if $Z_1$ and $Z_2$ are both non-empty, either $Z_1$ is completely blue to $B'$ or $Z_2$  is completely blue to $A''$. Without loss of generality, suppose that $Z_1$ is completely blue to $B'$. If now $|Z_2| \geq 1$, either there is at most one vertex in $Z_2$ with red neighbors in $A''$ or there is a vertex $a\in A''$ such that this is the only red neighbor of vertices in $Z_2$. Therefore, by removing at most one vertex from $V(G)$, it will also be completely blue between $Z_2$ and $A''$. 

In summary, we have disjoint sets $Z_1' \subset Z_1$, $Z_2' \subset Z_2$, $A''' \subset A''$, and $B'' \subset B'$ (at most one of which differs from its superset) such that $|Z_1' \cup Z_2'\cup A''' \cup B''| \geq n-1$ and the following conditions hold:
\begin{enumerate}
\item $Z_1'$ is completely red to $A'''$ and completely blue to $B''$.
\item $Z_2'$ is completely blue to $A'''$ and completely red to $B''$.
\item It is completely blue between $A'''$ and $B''$. 
\item $A'''$ and $B''$ are both red  cliques.
\end{enumerate}

Let $\tilde A = A''' \cup Z_1'$ and $\tilde B = B'' \cup Z_2'$. Then it is completely blue between $\tilde A$ and $B''$ and between $\tilde B$ and $A'''$. 
Furthermore, 
\[
|\tilde A| \geq |A'''| \geq |A'| - 2 \geq (1/2 -2 \sqrt{\lambda})n - 2 > (1/2 - 3 \sqrt{\lambda})n
\]
and, similarly, $|\tilde B| \geq (1/2 - 3\sqrt{\lambda})n$. Finally,  
\begin{equation}
|\tilde A| + |\tilde B| \geq n-1 = 2k-2.  \label{eqn:almost}
\end{equation}

By following the proofs of Claims~\ref{claim:red} and~\ref{claim:ABedge}, we obtain the following two results.

\begin{claim}\label{claim:tildeABblue}
If there is a blue edge within either $\tilde A$ or $\tilde B$, then there are at least $(n/5)^{k-2}$ blue cycles of length $k$.
\end{claim} 


\begin{claim} \label{claim:tildeABred}
Suppose that $\tilde A$ and $\tilde B$ are both red cliques, each with $k-1$ vertices. If there are two vertex-disjoint red edges between $\tilde A$ and $\tilde B$, then there are at least $(n/5e)^{k-6}$ red cycles of length $k$. 
\end{claim}


By Claim \ref{claim:tildeABblue}, 
we can assume that there is no blue edge within $\tilde A$ or $\tilde B$. That is, $\tilde A$ and $\tilde B$ are both red cliques. If either of these cliques has order at least $k$ we are done, as we then get at least $(k-1)!/2 \geq ((k-1)/2e)^{k-1}$ red cycles of length $k$. Hence, by (\ref{eqn:almost}), we can assume that $|\tilde A| = |\tilde B| = k-1$. 

By Claim \ref{claim:tildeABred}, 
we can assume that there are no two vertex-disjoint red edges between $\tilde A$ and $\tilde B$. Therefore, applying Claim \ref{claim:tworedABgeneral} to the red subgraph with $S = \tilde A$ and $T = \tilde B$, we see that by removing at most one vertex, say $w$, all the edges between $\tilde A$ and $\tilde B$ are blue. Without loss of generality, we will assume that $w \in \tilde A$, noting that $w$ must have at least one blue neighbor in $\tilde B$, since otherwise $\tilde B \cup \{w\}$ would be a red clique of order $k$, again completing the proof. 

We require one final observation, proved in the same manner as Claim~\ref{claim:PbothAB}.

\begin{claim} \label{claim:tildePbothAB}
Suppose that $w \in \tilde A$ and all the edges between $\tilde A \setminus \{w\}$ and $\tilde B$ are blue, while at least one edge between $w$ and $\tilde B$ is blue. If there is a blue path $P'$ of length two with one end vertex in $\tilde A$ and the other in $\tilde B$, then there are at least $(n/8e)^{k-5}$ blue cycles of length $k$. 
\end{claim}

Suppose that $u$ is the single vertex of $K_n$ which is not in $\tilde A \cup \tilde B$. If $u$ has a blue neighbor in both $\tilde A$ and $\tilde B$, then Claim \ref{claim:tildePbothAB} implies that we are done. Therefore, we can assume that $u$ is completely red to either $\tilde A$ or $\tilde B$. If $u$ is completely red to $\tilde A$, then $\tilde A \cup \{ u\}$ is a red clique with $k$ vertices, in which case there are at least $(k-1)!/2$ red cycles of length $k$. Since this is also true if $u$ is completely red to $\tilde B$, this completes the proof.
\end{proof}

\subsection{Proof of Lemma \ref{lem:main2}} \label{sec:stab}

The following stability lemma of Nikiforov and Schelp \cite{NS} is an essential ingredient in our proof. 

\begin{lemma}[Theorem 13, \cite{NS}]\label{lem:NS}
Let $0 < \alpha < 5 \cdot 10^{-6}$, $0 \leq \beta \leq \alpha/25$, and $n \geq \alpha^{-1}$.
If $G$ is a graph with $n$ vertices and $e(G) > (1/4 - \beta) n^2$, then one of the following holds:
\begin{enumerate}
\item There are cycles $C_t \subset G$ for every $t \in [3, \lceil(1/2 + \alpha) n\rceil]$.
\item There exists a partition $V(G) = U_0 \cup U_1 \cup U_2$ such that
\[|U_0| < 2000\alpha n,\]
\[ \left( 1/2 - 10\sqrt{\alpha+\beta}\right) n < |U_1| \leq |U_2| <  \left( 1/2 + 10\sqrt{\alpha+\beta}\right) n,\]
and the induced subgraph $G- U_0$ on vertex set $V(G) \setminus U_0$ is a subgraph of either the complete bipartite graph between $U_1$ and $U_2$ or its complement. 
\end{enumerate}
\end{lemma}

With this preliminary in place, we can begin the proof of Lemma \ref{lem:main2}. We apply the colored regularity lemma, Lemma~\ref{reglem-twocolor}, with parameters $\epsilon$ and $l = \lceil \epsilon^{-1} \rceil$ to the given red/blue coloring of $K_n$. This implies that there exist $n_0(\epsilon)$ and $M_0(\epsilon)$ such that, for any $n \geq n_0$, there is a regular partition of $K_n$ into $M$ parts $V_1, \dots, V_M$ with $\epsilon^{-1} \leq M \leq M_0$. We now consider a reduced graph $H$ with $M$ vertices $v_1, \dots, v_M$ corresponding to $V_1, \dots, V_M$, placing an edge between $v_i$ and $v_j$ if and only if the pair $(V_i, V_j)$ is $\epsilon$-regular. We then color the edge $(v_i, v_j)$ red if the density of red edges between $V_i$ and $V_j$ is at least $d = 12 \epsilon^{1/2}$ and we color an edge blue under the analogous condition with blue in place of red. By the regularity lemma, all but at most $\epsilon \binom{M}{2}$ pairs of distinct vertices of $H$ are edges and, since $d < 1/2$, all edges of $H$ are colored red, blue, or, perhaps, both red and blue.  We say an edge is red-only if it is colored in red and not blue, while blue-only is defined similarly. 

Let the subgraph of $H$ induced by edges containing the color red be $H_R$ and the subgraph induced by edges containing the color blue be $H_B$. Hence, 
\[ |E(H_R) | +  |E(H_B) |  \geq (1-\epsilon)\binom{M}{2} > (1 - 2\epsilon) M^2/2,\]
where we used that $M \geq \epsilon^{-1}$. Thus, without loss of generality, we can assume that
\[ |E(H_R) | >  (1-2\epsilon)M^2/4.\] 

We now apply Lemma~\ref{lem:NS} to $H_R$ with $\beta = \epsilon/2$ and $\alpha = 20 \sqrt{\epsilon}$. There are two cases: 

\paragraph{Case 1 of Lemma~\ref{lem:NS}.}
In this case, we can find a red cycle $C_t$ for every $t \in [3, \lceil(1/2 + \alpha) M\rceil]$. In particular, we can find an odd cycle $C_t$ with 
\[ (1/2 + \alpha) M \geq t > (1/2 + \alpha) M - 2.\]
But this means that there are disjoint vertex sets $V_{k_0}, \dots, V_{k_{t-1}}$ such that, for each $0 \leq i \leq t-1$, $|V_{k_i}| \geq  \lfloor n/M \rfloor$ and each pair $(V_{k_i}, V_{k_{i+1}})$ (with addition taken mod $t$) is $\epsilon$-regular in red with red density at least $12 \epsilon^{1/2}$.
Thus, we are in Case 1 of Lemma \ref{lem:main2}.

\paragraph{Case 2 of Lemma~\ref{lem:NS}.}
In this case, there exists a partition $V(H_R) = U_0 \cup U_1 \cup U_2$ such that $|U_0| < 2000\alpha M$ and
\begin{equation} 
\left( 1/2 - 10\sqrt{2\alpha}\right) M < |U_1| \leq |U_2| <  \left( 1/2 + 10\sqrt{2\alpha}\right) M.  \label{eqn:U1U2}
\end{equation}
Furthermore, the induced subgraph $H_R-U_0$ is a subgraph of the disjoint cliques on $U_1$ and $U_2$ or a subgraph of the complete bipartite graph between $U_1$ and $U_2$. We will assume that the induced subgraph $H_R-U_0$ is a subgraph of the graph consisting of disjoint cliques on $U_1$ and $U_2$. The other case, where all edges of $H_R - U_0$ are between $U_1$ and $U_2$, can be handled similarly. 

Thus, by assumption, any edges between $U_1$ and $U_2$ are blue-only. Moreover, since the number of non-adjacent pairs is at most $\epsilon \binom{M}{2}$, the number of blue-only edges between $U_1$ and $U_2$ is at least
\beq
|U_1||U_2| - \epsilon \binom{M}{2}. \label{eqn:U1U2blue}
\eeq
Let $U_1' \subset U_1$ be the set of vertices in $U_1$ that have blue degree at least $(1-\sqrt{\epsilon}) |U_2|$ in $U_2$. 
Suppose $|U_1 \setminus U_1'| = x |U_1|$. Then
\[
(1-\sqrt{\epsilon} )|U_2| x |U_1| + |U_2|(1-x)|U_1| \geq |U_1||U_2| - \epsilon M^2/2,
\]
which implies that $x \leq \sqrt{\epsilon}M^2 / (2|U_1||U_2|)$.
Since $|U_1|, |U_2| \geq  (1/2 - 10\sqrt{2\alpha}) M$, we have 
\[
 x \leq \sqrt{\epsilon}M^2 / (2|U_1||U_2|) \leq \sqrt{\epsilon}M^2 / (2  (1/2 - 10\sqrt{2\alpha})^2 M^2   )< \sqrt{\epsilon}(2 + 200\sqrt{\alpha}) < 3 \sqrt{\epsilon},
 \]
where we used that $\alpha < 5\cdot 10^{-6}$. Defining $U_2' \subset U_2$ analogously, we therefore have 
 \begin{equation}
 |U_1 \setminus U_1'| \leq 3\sqrt{\epsilon}|U_1|, \ \  |U_2 \setminus U_2'| \leq 3\sqrt{\epsilon}|U_2|.  \label{eq:U1U1'}
 \end{equation}
 Thus, each vertex in $U_1'$ has at least $(1-\sqrt{\epsilon})|U_2| - |U_2 \setminus U_2'| \geq (1-4\sqrt{\epsilon})|U_2|$ blue neighbors in $U_2'$ and, similarly, each vertex in $U_2'$ has at least $(1-4\sqrt{\epsilon})|U_1|$ blue neighbors in $U_1'$. 
 
 \begin{claim}\label{claim:U1U2red'}
If there is a blue edge within $U_1'$ or $U_2'$, then Case 1 of Lemma \ref{lem:main2} holds. 
 \end{claim}

 \begin{proof}
It will suffice to show that there is a blue cycle $C_t$ in $H$, where $t$ is the odd integer with $(1/2 + \alpha)M \geq t > (1/2+\alpha)M-2$. Suppose that there is a blue edge $(u,u')$ in $U_1'$. We will apply Claim \ref{claim:redgeneral} to $H_B$ with $(S, T, l)$ being $(U_1', U_2', t)$. Since each vertex in $U_1'$ has at least $ (1-4\sqrt{\epsilon})|U_2|$ blue neighbors in $U_2'$, any two vertices in $U_1'$ have blue common neighborhood in $U_2'$ of order at least 
 \[
 2(1-4\sqrt{\epsilon})|U_2| - |U_2'| \geq (1-8\sqrt{\epsilon})|U_2'|. 
 \]
 Thus, we can let $s$ in Claim \ref{claim:redgeneral} be $(1-8\sqrt{\epsilon})|U_2'|$. 
 To check that Claim \ref{claim:redgeneral} applies, we need to show that $(1/2 + \alpha)M \leq \min(2|U_1'| -1, 2s+1)$. 
 First, by (\ref{eqn:U1U2}), (\ref{eq:U1U1'}), and the fact that $M \geq \epsilon^{-1}$,
 \[
2|U_1'|-1 \geq  2 (1-3\sqrt{\epsilon})|U_1| -1\geq 2(1-3\sqrt{\epsilon})  (1/2 - 10\sqrt{2\alpha})M -1 > 0.6M >  (1/2+\alpha)M. 
 \]
 Similarly, 
  \[
2s+1 \geq  2(1-8\sqrt{\epsilon}) (1-3\sqrt{\epsilon})|U_2| \geq 2(1-11\sqrt{\epsilon})  (1/2 - 10\sqrt{2\alpha})M > 0.6M >  (1/2+\alpha)M. 
 \]
Thus, by Claim \ref{claim:redgeneral}, there is a cycle of length $t$ in $H_B$, as required. 
\end{proof}

We may therefore assume that there is no blue edge inside $U_1'$ or $U_2'$. That is, all the edges within $U_1'$ and $U_2'$ are red-only. We move the vertices in $U_0$ arbitrarily to $U_1$ and $U_2$ to obtain $\tilde U_1$ and $\tilde U_2$. Thus, $\tilde U_1 \cup \tilde U_2$ is a vertex partition of $V(H)$.
Let $X_1 \subset V(K_n)$ be the vertices in $K_n$ corresponding to $\tilde U_1$ in $H$ and let $X_2 \subset V(K_n)$ be the vertices corresponding to $\tilde U_2$. We will conclude the proof by showing that the partition $X_1 \cup X_2$ induces an extremal coloring. 

\begin{claim}
The vertex partition $X_1 \cup X_2$ induces an extremal coloring with parameter $\lambda$, where $\lambda = 300 \sqrt{\alpha}$.
\end{claim}

\begin{proof}
By Claim \ref{claim:U1U2red'}, any edge in $U_1'$ is red-only and at most $\epsilon \binom{M}{2}$ pairs of distinct vertices in $U_1'$ are non-adjacent. 
Moreover, for any red-only edge $(i,j)$ in $U_1'$, the red density between $V_i$ and $V_j$ is at least $1-d$, since otherwise $(i,j)$ would also be colored blue. Since $n/M - 1 < |V_i| \leq n/M+1$, the number of red edges in $X_1$ is at least
\[(1-d) (n/M-1)^2 \binom{|U_1'|}{2} - (n/M+1)^2 \epsilon \binom{M}{2}.\]
Note also, by (\ref{eqn:U1U2}), that
\begin{align*}
|X_1| & \leq |\tilde U_1| \cdot (n/M+1) \leq (|U_1|+|U_0|)(n/M+1) \leq (|U_1| + 2000\alpha M) (n/M+1) \nonumber \\
& \leq \left( 1/2 + 10\sqrt{2\alpha} + 2000
\alpha \right) M (n/M+1) 
\leq (1/2 + 20\sqrt{\alpha}) n.
\end{align*} 
Combining the two inequalities above with (\ref{eqn:U1U2}) and (\ref{eq:U1U1'}), 
we see that the red density in $X_1$ is at least 
\begin{align*}
\frac{(1-d) \binom{|U_1'|}{2} (n/M-1)^2 - (n/M+1)^2\epsilon M^2/2}{ |X_1|^2/2} 
& \geq  \frac{2(1-d) \binom{(1-3\sqrt{\epsilon})|U_1|}{2} (n/M-1)^2 - (n/M+1)^2\epsilon M^2}{  (1/2 + 20\sqrt{\alpha})^2 n^2   } \\
& \geq  1 - d - 200 \sqrt{\alpha} - 10\sqrt{\epsilon} > 1 - 300 \sqrt{\alpha}. 
\end{align*}
Similarly, $|X_2| \leq (1/2 + 20\sqrt{\alpha}) n$ and the red density in $X_2 \subset V(G)$ is at least $1 - 300 \sqrt{\alpha}$. 

It only remains to lower bound the blue density between $X_1$ and $X_2$. By (\ref{eqn:U1U2}) and (\ref{eqn:U1U2blue}), the number of blue edges between $X_1$ and $X_2$ is at least
\begin{align*}
(1-d)\left(|U_1||U_2| - \epsilon \binom{M}{2}\right) (n/M-1)^2 
& \geq (1-d)\left((1/2 - 20\sqrt{\alpha})^2 M^2 - \epsilon \binom{M}{2}\right) (n/M - 1)^2 \\ 
& >(1-d) (1/4 - 25 \sqrt{\alpha})n^2.
\end{align*}
Thus, by a similar computation to before, the blue density between $X_1$ and $X_2$ is at least
\[\frac{(1-d)(1/4 - 25 \sqrt{\alpha})n^2}{|X_1||X_2|} 
\geq \frac{(1-d)(1/4 - 25 \sqrt{\alpha})n^2}{n^2/4} >  1 - 300\sqrt{\alpha},\]
as required.
\end{proof}

\end{document}